\documentclass{article}
\usepackage{comment}
\usepackage{graphicx} 
\usepackage{xcolor} 
\usepackage[english]{babel}
\usepackage[utf8]{inputenc}
\usepackage[T1]{fontenc}
\usepackage{amsmath}
\usepackage{amssymb}
\usepackage{amsthm}
\usepackage{amsfonts}
\usepackage{tikz}
\usetikzlibrary{calc}
\usepackage{standalone}
\usetikzlibrary{decorations.pathreplacing}

\usepackage{xcolor}
\usepackage[verbose]{hyperref}
\hypersetup{colorlinks=false,allbordercolors=blue,pdfborderstyle={/S/U/W 1}}
\usepackage{standalone}
\usepackage{tikz}
\usepackage{comment}

\def\exp{\mathrm{exp}}

\def\|{\mid}

\begin{document}

\newtheorem{proposition}{Proposition}[section]
\newtheorem{question}[proposition]{Question}
\newtheorem{note}[proposition]{Note}
\newtheorem{theorem}[proposition]{Theorem}
\newtheorem{corollary}[proposition]{Corollary}
\newtheorem{lemma}[proposition]{Lemma}
\newtheorem{consequence}[proposition]{Consequence}
\theoremstyle{definition}
\newtheorem{definition}[proposition]{Definition}
\newtheorem{example}[proposition]{Example}
\newtheorem{remark}[proposition]{Remark}
\newtheorem{claim}[proposition]{Claim}
\newtheorem{construction}[proposition]{Construction}
\newtheorem{assumption}[proposition]{Assumption}

\markboth{Dominika Bure\v{s}ov\'{a}, Pavel Pt\'{a}k, Jan \v{S}evic}{Simultaneous extensions of states on quantum logics}

\author{Dominika Bure\v{s}ov\'{a}, Pavel Pt\'{a}k, Jan \v{S}evic}

\title{Simultaneous extensions of states on quantum logics}




\maketitle

\begin{abstract}
   Suppose that $L$ is a quantum logic (= an orthomodular poset) and
$A,B$ are Boolean subalgebras of $L$. Suppose further that
$s_1(\mathrm{resp.}\; s_2)$ is a state on $A$
(resp. on $B$). We ask when both $s_1$ and $s_2$
allow for a simultaneous extension over $L$.
Under a natural necessity condition on
$A,B,s_1,s_2$ we find a sufficient condition
for the existence of the extension.
This condition is the requirement of an abundance
of two-valued states on $L$. We then show
by the Greechie pasting technique that this condition is far from necessary even for state-space-rich logics.
\end{abstract}



\section{Introduction}

As conceptually adopted in the description of the event structure of a quantum experiment, the structure is assumed to be an orthomodular partially ordered set. The event structure is then called a \textit{quantum logic} (abbr. a logic, see e.g.~{\cite{GudderBook,PtakPulmannova, Svozil}, etc.).
The states of the experiment are then associated with probability measures on the logic. The hitherto effort to understand the behavior of states on logics brought a~series of results (see \cite{HornTarski, Kalmbach, Greechie, NavaraPtakRogalewicz, PtakExotic, Shultz}, etc.).
One of natural questions motivated by theoretical physics can roughly be posed as follows.
If we have two standard sub-experiments of a~quantum experiment and two states on them, when these two states can be "subordinated" to a~single state on the entire quantum experiment? Expressed in the formalism of quantum logics, if $A,B$ are Boolean subalgebras of a quantum logic $L$ and $s_1,s_2$ are states on $A,B$, is there a state $s$ on $L$ with $s(a)=s_1(a)$ for any $a\in A$ and $s(b)=s_2(b)$ for any $b\in B$?
In this note we address this question. Applying the results of~\cite{RaoRao}, we find a sufficient condition for a positive answer to this question. 
By using the Greechie paste job, we indicate potential generalizations for simultaneous extensions of states.

\section{Notions and results}

Let us first recall the notions we shall use in the sequel.

\begin{definition}
By a quantum logic (abbr.\ by a logic) we mean a $4$--tuple
$(L,\leq,{}',0,1)$, where $L$ is a partially ordered set with the relation
$\leq$ and with operation ${}' : L\to L$ such that $(a,b \in L)$

\begin{enumerate}
\item[(i)] $a\leq b \Rightarrow b'\leq a'$,

\item[(ii)] $a\vee a'=1$ and $a\wedge a'=0$,

\item[(iii)] $a''=a$,

\item[(iv)] if $a\leq b$, then
$
b=a\vee (b\wedge a').
$
\end{enumerate}
\end{definition} \bigskip

We shall need some more notions. If $A$ is a subset of $L$ such that
$A$ is a Boolean algebra with the
operations inherited from $L$,
then we call $A$ a \textit{Boolean sublogic} of $L$.

Let $s:L\to\langle 0,1\rangle$ be a mapping. Then $s$ is said to be a \textit{state} on $L$ if

\begin{enumerate}
\item $s(1)=1$,

\item $s(a\vee b)=s(a)+s(b)$ provided
      $a\leq b'$.
\end{enumerate}

Denote by $\mathcal{S}(L)$ (resp. $\mathcal{S}_2(L)$)  the set of all
states on $L$ (resp. the set of all $2$-valued states on $L$). As known, $\mathcal{S}(L)$ may be
quite small even for big logics (\cite{Greechie}).
Obviously, $\mathcal{S}(L)$ viewed as a subset of the product of $[0,1]$ is convex and compact.
The central notions of this paper
read as follows.

\begin{definition}

Let $L$ be a~logic and let $A$ and $B$
be Boolean subalgebras of $L$. Let
$s_1$ (resp. $s_2$) be a state on $A$
(resp. on $B$).

We say that $A,B,s_1,s_2$ are \textit{well conditioned} provided

\begin{enumerate}
\item if $a\in A$ and $b\in B$ and $a\ge b$,
      then $s_1(a)\ge s_2(b)$,

\item if $c\in A$ and $d\in B$ and $c\le d$,
      then $s_1(c)\le s_2(d)$.
\end{enumerate}
\end{definition}

Let us call $s_1$ and $s_2$ \textit{simultaneously
extensible} whenever there is a state $s$
on $L$ such that $s(a)=s_1(a)$ for any $a \in A$  and $s(b)=s_2(b)$ for any $b \in B$. Obviously, the requirement of being
well conditioned is a necessary condition
for the simultaneous extensibility. Also, if $a \in A \cap B$, then $s_1(a)=s_2(a)$.

Finally, let us say that $L$ allows for \textit{simultaneous extensions of states} if any well-conditioned quadruple $A,B,s_1,s_2$ of $L$ is simultaneously extensible. Let us first recall how the situation looks like in the fully ``standard'' setting
(i.e. in the case when $L$ is Boolean). The answer is given by the
following remarkable result presented
as Th. 3.6.1 in~\cite{RaoRao} (see also~\cite{HornTarski}).

\begin{theorem}
Let $L$ be a Boolean algebra. Then $L$ allows for simultaneous extensions of states.
\end{theorem}

The main objective of this paper is the generalization of this theorem to quantum logics. Before formulating our first result, let us recall the notion of a~set-representable logic.
\begin{definition}
    Let $Q$ be a~set and let $\mathcal{L}$ be a~collection of subsets of $Q$. Let us suppose that $\mathcal{L}$ fulfills the following properties:
    \begin{enumerate}
        \item $Q \in \mathcal{L}$,
        \item if $T, U \in \mathcal{L}$ and $T \cap U = \emptyset$, then $T \cup U \in \mathcal{L}$,
        \item if $T \in \mathcal{L}$, then $Q \setminus T \in \mathcal{L}$.
    \end{enumerate}
    Let us assume that $f \in L \to \mathcal{L}$ is such a~mapping that 
    \begin{enumerate}
        \item $f: L \to \mathcal{L}$ is a~set isomorphism,
        \item both $f$ and $f^{-1}$ are logic morphisms.
    \end{enumerate}
    Then $Q$ is said to be a~\textit{set-representable logic}. 
\end{definition}
Obviously, a~set-representable logic is a logic. The set-representable logics present an important class of logics (see e.g.,~\cite{BuresovaPtak, GudderPaper, PtakConcrete, Svozil}, etc.). It may be noted that any bounded lattice can be lattice-embedded in a~set-representable logic (see~\cite{Kalmbach, HardingNavara}).
Let us formulate our first result.

\begin{theorem}\label{Th1}
Let $L$ be a quantum logic and let $a,b \in L$. Let the state space of $L$ have the following property: If $a\not\leq b$, then there is a two-valued state $s\in \mathcal{S}(L)$ such that $s(a)=1$ and $s(b)=0$. Then

\begin{enumerate}
\item $L$ is a set-representable logic,

\item $L$ allows for simultaneous extensions of states.
\end{enumerate}

\end{theorem}

\begin{proof}
(1) This part can be easily obtained by mimicking the Stone representation theorem for Boolean algebras. Let $S=S_2(L)$
and
$$
\mathcal L=
\left\{A \mid A \subseteq S, A= \{s\in S\mid s(a)=1 \}~\text{for some }
a\in L
\right\}
$$

Then $\mathcal L$ is a~logic of subsets of $S$ provided the relation $\leq$ is the inclusion ordering in $S$ and the complementation $'$ in $\mathcal L$ is the set-orthocomplementation $(A'=S \setminus A)$.
It is a matter of a~simple Boolean-like verification to check that the mapping $f:L\to\mathcal L$ 
defined by
$
f(a)=\{\,s\in S_2(L)\mid s(a)=1\,\}
$
is a logic isomorphism. Hence $L$ is set-representable. We can therefore assume that $L=\mathcal L$, where $\mathcal L$ is a
collection of subsets of $S$.

It is important to realize that $A$ (resp.\ $B$) is a~Boolean subalgebra of $L$ precisely
when $A$ (resp.$B$) is closed under the formation in $L$ of unions, intersections and
complements (see e.g.,~\cite{PtakConcrete}.

Let $A,B,s_1,s_2$ be a~well conditioned quadruple in $L$. Let $\exp(S)$ be the collection of all subsets of $S$ end let us view $\exp(S)$ as a Boolean algebra. By the previous argument, $A$ and $B$ remain Boolean subalgebras of
$\exp(S)$. By Th.~\ref{Th1} there is an extension of $s_1$ and $s_2$ over $\exp(S)$. By restricting $s$ to $L$ we obtain the required extension of $s_1$ and $s_2$ over $L$. The proof is complete.
\end{proof}

The following consequence of Th.~\ref{Th1} is worthwhile formulating. It represents a~notable result in its own right. Let us call $L$ \textit{compatibility regular} if
$a\wedge b'=0 \iff a\leq b$. Recall that all compatibility regular logics enjoy the following
property that "algebraizes" the notion of compatibility : If $a,b\in L$, then $a,b$ are compatible (i.e. $a,b$ generates a Boolean subalgebra of $L$) in $L$ exactly when $a \land b$ exists in $L$.

\begin{theorem}\label{Th2}
If $L$ is compatibility regular, then $L$ allows for simultaneous extensions of states.
\end{theorem}

\begin{proof}
By the main result of \cite{NavaraPtak},
$L$ is set-representable and Theorem~\ref{Th2} applies.\\

\end{proof}

Let us discuss Th.~\ref{Th1} and its potential generalizations.
First, since any logic can be embedded in a~stateless logic~\cite{Greechie, NavaraPtakRogalewicz}, any logic can be embedded in a~logic that does not allow for simultaneous extensions of states.
As regards a~generalization of Theorem~\ref{Th1}, a~natural
limit is formulated in the following
observation. It disqualifies
several important logics like $L(H)$
(the logic of projections in a Hilbert
space $H$) for simultaneous extensions of states.

\begin{proposition}\label{Prop1}
Suppose that $a,b\in L$ and $a\not\leq b$.
Suppose that there is no $s\in \mathcal{S}(L)$ with
$s(a)=1$ and $s(b)=0$. Then $L$ does not allow for simultaneous extensions of states.
\end{proposition}

\begin{proof}
Take two $4$--element Boolean subalgebras $A$ and $B$ of $L$
such that $a\in A$ and $b\in B$. Suppose that $s_1\in S_2(A)$ and $s_2\in S_2(B)$ are such that
$s_1(a)=1$ and $s_2(b)=0$. Then $A,B,s_1,s_2$ are well conditioned in $L$ and there is no
state $s\in \mathcal{S}(L)$ that extends both $s_1$ and $s_2$.
\end{proof}

By the previous observation, a~natural relaxation of the condition of Th.~\ref{Th1} seems to be the property of being \textit{$0$--$1$ distinguishing}. Let us say that $L$ is $0$--$1$ distinguishing if for any $a,b\in L$ with $a\not\leq b$ there is a state $s\in \mathcal{S}(L)$ such that
$s(a)=1$ and $s(b)=0$ ($s$ is not required to be two-valued). 
Though we haven't been able to prove Theorem~\ref{Th1} for all the logics with the $0$--$1$ distinguishing state space, the following two results suggest it can be called a~conjecture. The first one concerns the case when $A=B$.

\begin{proposition}\label{Prop2}
Let $A$ be a Boolean subalgebra of a logic $L$. Let $L$ be $0$--$1$ distinguishing. Then any state $s\in \mathcal{S}(A)$ can be extended over $L$ as a~state.
\end{proposition}

\begin{proof}
Let $\mathcal F$ be the collection of all finite subalgebras of $A$. If
$F=\{0,1,a_1,\dots,a_n\}\in\mathcal F$,
then there is a state $v_F\in \mathcal{S}(L)$ such that $v_F$ restricted to $F$ equals to $s$.
Indeed, for each $a_i$ there is a state $t_i$ on $L$ with
$t_i(a_i)=1$ and we set $v_F=\sum_{i=1}^{n}s(a_i)t_i$.
Since $\sum_{i \leq n}s(a_i)=1$, we see that $v_F\in \mathcal{S}(L)$.

The collection of finite subalgebras of $A$ is a directed set under inclusion, so the compactness of $\mathcal{S}(L)$ gives us a state
$t$ that we obtain as a~directed limit $t=\lim\limits_{F\in\mathcal F}v_F$.
Obviously, $t$ extends $s$.
\end{proof}

The second example contributes to the orthomodular combinatorics at large. As a~byproduct, it generates a~variety of "almost set-representable orthomodular lattices" (see~\cite{Mayet, MayetPtak}).

\begin{proposition}
There is a finite (lattice) logic $L$ such that

\begin{enumerate}
\item $L$ is not set-representable,
\item $L$ is $0$--$1$ distinguishing,
\item $L$ allows for a simultaneous extension of states.
\end{enumerate}
\end{proposition}

\begin{proof}
We make use of the Greechie "paste job" (\cite{Greechie}). 
It became a~folklore in orthomodular structures.
Recall only that in the
diagram below the unbroken lines represent the maximal Boolean
subalgebras of $L$ (in our case these are $3$-atom and
$4$-atom Boolean subalgebras), and the subalgebras meet in
a~single atom as depicted in the figure. An important property of the logic $L$ defined by the
diagram is the determination of states on $L$. A state is
precisely an evaluation of the vertices of the diagram
that sums up to $1$ over any maximal Boolean segment of
the diagram.

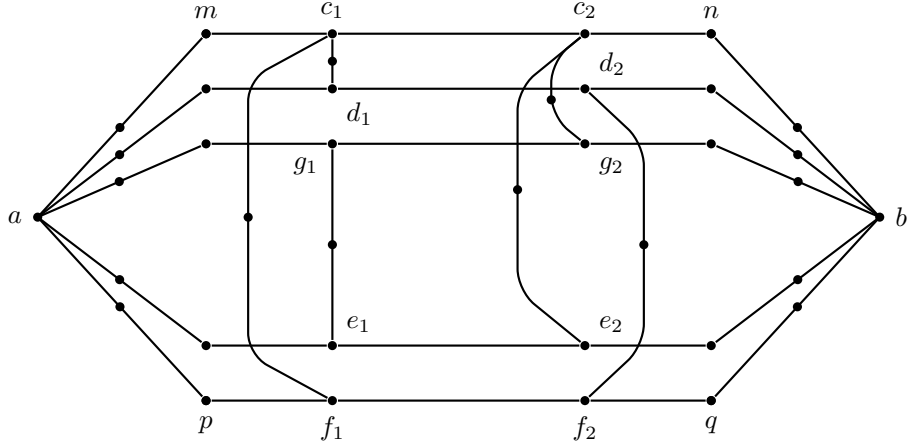
\begin{tikzpicture}[
    dot/.style={circle, fill=black, inner sep=1.2pt},
    poly/.style={rounded corners=8pt},
    thick,
    >=stealth,
    xscale=1.1,
    yscale=1.2
]

\coordinate (a) at (0, 3.0);
\coordinate (b) at (10.0, 3.0);

\node[dot, label=left:$a$] at (a) {};
\node[dot, label=right:$b$] at (b) {};

\foreach \y [count=\i] in {5.0, 4.4, 3.8, 1.6, 1.0} {
    \node[dot] (CornerL\i) at (2.0, \y) {};
    \draw (a) -- node[dot] {} (CornerL\i);
}
\node[dot, label=90:$m$] (CornerL1) at (2.0, 5.0) {};
\node[dot, label=-90:$p$] (CornerL5) at (2.0, 1.0) {};

\foreach \y [count=\i] in {5.0, 4.4, 3.8, 1.6, 1.0} {
    \node[dot] (CornerR\i) at (8.0, \y) {};
    \draw (CornerR\i) -- node[dot] {} (b);
}
\node[dot, label=90:$n$] (CornerR1) at (8.0, 5.0) {};
\node[dot, label=-90:$q$] (CornerR5) at (8.0, 1.0) {};

\node[dot, label=90:$c_1$]  (c1) at (3.5, 5.0) {};
\node[dot, label=-45:$d_1$] (d1) at (3.5, 4.4) {};
\node[dot, label=-135:$g_1$](g1) at (3.5, 3.8) {};
\node[dot, label=45:$e_1$]  (e1) at (3.5, 1.6) {};
\node[dot, label=-90:$f_1$] (f1) at (3.5, 1.0) {};

\draw (CornerL1) -- (c1);
\draw (CornerL2) -- (d1);
\draw (CornerL3) -- (g1);
\draw (CornerL4) -- (e1);
\draw (CornerL5) -- (f1);

\draw (c1) -- node[dot] {} (d1);                                         
\draw (g1) -- node[dot] {} (e1);        

\draw[poly] (c1) -- (2.5, 4.5) -- node[dot, pos=0.5] {} (2.5, 1.5) -- (f1); 

\node[dot, label=90:$c_2$]  (c2) at (6.5, 5.0) {};
\node[dot, label=45:$d_2$]  (d2) at (6.5, 4.4) {}; 
\node[dot, label=-45:$g_2$] (g2) at (6.5, 3.8) {};
\node[dot, label=45:$e_2$]  (e2) at (6.5, 1.6) {};
\node[dot, label=-90:$f_2$] (f2) at (6.5, 1.0) {};

\draw (c1) -- (c2);
\draw (d1) -- (d2); 
\draw (g1) -- (g2);
\draw (e1) -- (e2);
\draw (f1) -- (f2);


\draw[poly] (c2) -- (5.7, 4.4) -- node[dot, pos=0.5] {} (5.7, 2.2) -- (e2); 
\draw[poly] (c2) -- (6.1, 4.7) -- node[dot, pos=0.7] {} (6.1, 4.1) -- (g2); 
\draw[poly] (d2) -- (7.2, 3.8) -- node[dot, pos=0.5] {} (7.2, 1.6) -- (f2); 

\draw (c2) -- (CornerR1);
\draw (d2) -- (CornerR2);
\draw (g2) -- (CornerR3);
\draw (e2) -- (CornerR4);
\draw (f2) -- (CornerR5);

\end{tikzpicture}

Observe first that the logic determined by the diagram is
not set-representable. Indeed, suppose that there is a
state $s\in S_2(L)$ such that
$s(a)=1$ and $s(b')=0$. Therefore, $s(b)=1$ and all vertices sharing an edge with $a$ or $b$ have a state value of $0$ in the state $s$.

Obviously, $s(m)+s(c_1) + s(c_2)+s(n)=s(c_1) + s(c_2)=1$.
Suppose that $s(c_1)=1$. Then $s(d_1)=0$ and $s(f_1)=0$. This means that $s(f_2)=1$
and therefore $s(d_2)=0$. We obtain that
$s(d_1)+s(d_2)=0$ which is a~contradiction.
Analogously, if we assume that $s(c_2)=1$, then $s(g_2)=0$ and $s(e_2)=0$. This means that $s(e_1)=1$ and therefore $s(g_1)=0$. Then 
$s(g_1)+s(g_2)=0$ and that is again a~contradiction.

It is easy to show that $L$
is $0$--$1$ distinguishing. The critical
case is again $s(a)=1,\qquad s(b')=0$,
otherwise we can even distinguish by a state of $S_2(L)$.
In the critical case we assign
$s(c_1)=\frac12$ and $s(c_2)=\frac12$.
Analogously for $d_1, d_2, e_1, e_2$, $f_1,f_2$, and $g_1,g_2$. Thus we obtain the required state in this way. 

Finally, the logic $L$ allows for simultaneous extensions of states. There
are many cases to be checked but they
are easy to verify. The only case that may jeopardize the verification is this:
$A$ is the block that contains $a,m$ and $B$ is the block that contains $b,n$, and $s_1(m)=1$ and $s_2(n)=1$. But then $A, B, s_1,s_2$ is not well-conditioned.
\end{proof}

\section*{Acknowledgements}
This research was supported by the GACR (project No. 25-20013L) and the SGS (project No. SGS26/073/OHK3/1T/13).

\textbf{Data availability statement.} The manuscript has no associated data.

\textbf{Conflict of interest statement.} The authors state that there is no conflict of interest. 

\section*{ORCID}
\noindent Dominika Burešová - \url{https://orcid.org/0000-0002-8760-5015}

\noindent Pavel Pták - \url{https://orcid.org/0000-0003-4949-8236}

\noindent Jan Ševic - \url{https://orcid.org/0009-0004-2313-7373}

\end{document}